\documentclass[letterpaper]{amsart}

\usepackage[english]{babel}
\usepackage{amsmath}
\usepackage{graphicx}
\newtheorem{theorem}{Theorem}[section]
\newtheorem{lemma}[theorem]{Lemma}

\newtheorem{corollary}[theorem]{Corollary}

\theoremstyle{definition}

\newcommand{\E}{\mathop{\bf E\/}}

\title {A Simple Proof of the Cayley Formula using Random Graphs}
\author[S. Wu]{Scott Wu}
\author[R. Li]{Ray Li}
\author[A. He]{Andrew He}
\author[S. Hao]{Steven Hao}

\begin{document}

\begin{abstract}
We present a nice result on the probability of a cycle occuring in a randomly generated graph. We then provide some extensions and applications, including the proof of the famous Cayley formula, which states that the number of labeled trees on $n$ vertices is $n^{n-2}.$ \cite{Cayley} 
\end{abstract}

\maketitle

\section {Proof of Cayley's Theorem}

\begin {lemma}
\label{keylemma}
Let $G$ be a graph with $n$ vertices, initially empty, and let $p \in [0,1]$. For each vertex $v \in G$, indepedently perform the following process: Pick a vertex $w$ uniformly at random from the vertices of $G$. With probability $p$, call $v$ \textit{good} and create a directed edge from $v$ to $w$.
The probability that $G$ contains a directed cycle of any kind, including self-loops, is $p$.
\end {lemma}

\begin {proof}
We proceed by strong induction on $n$. If $n=1$, the result is trivial.

Consider a graph $G$ on $n \ge 2$ vertices generated as stated above. Suppose that $k$ of the $n$ vertices are good. We then claim that the probability that a cycle exists in the graph is $\frac{k}{n}$. 

If $k = n$, then the graph clearly contains a cycle, so the probability is $1 = \frac{n}{n}$. 
If $k = 0$, the graph clearly does not contain a cycle, so the probability is $0$. 
Now suppose that $0 < k < n$. 
%Let $G'$ be the subgraph of $G$ that consists of the $k$ good vertices.
Note that any cycle in the graph must consist only of good vertices.
Therefore, the probability of the existence of a cycle in $G$ is equal to the probability of the existence of a cycle in the subgraph $G'$ consisting of the $k$ good vertices. Each vertex in $G'$ has probability $\frac{k}{n}$ of having an outedge leading to a vertex in $G'$. %"outedge"
By the inductive hypothesis, $G'$, and therefore $G$, has probability $\frac{k}{n}$ of containing a cycle, as desired.

We have shown that if $k$ of the $n$ vertices are good, the probability of a cycle existing is $\frac{k}{n}.$ Thus, the expected probability of a cycle is $\E \left[\frac{k}{n}\right]$, which is $\frac{\E [k]} {n} = \frac{p\cdot n}{n} = p$ by linearity of expectation. Thus, the probability of creating a cycle is $p$ for $n$ vertices, and the induction is complete. 
\end {proof}

\iffalse
\begin{corollary}
The underlying undirected graph $G'$ for the graph $G$ generated in Lemma~\ref{keylemma} contains a cycle with probability $p.$
\end{corollary}
\begin{proof}
We claim that $G'$ has a cycle if and only if $G$ has a cycle. If $G$ has a cycle, $G'$ clearly does as well. If $G'$ has a cycle, the corresponding edges in $G$ must also form a cycle because every vertex of $G$ has outdegree at most 1. The result follows. 
\end{proof}
\fi

\begin {theorem}
(Cayley's Formula) For any positive integer $n$, the number of trees on $n$ labeled vertices is exactly $n^{n-2}$. 
\end {theorem}

\begin{proof}
Assume $n \ge 2$. If $n = 1$ the proof is trivial.

Apply Lemma~\ref{keylemma} on a graph of $n-1$ vertices with $p = \frac{n-1}{n}$. Each vertex is equally likely to have a directed edge to a specified vertex or to have no outedge at all, so there are $n^{n-1}$ possible directed graphs, each equally probable. Call one such graph \textit{special} if it does not contain any cycles. Note that special graphs must be forests of rooted trees with edges pointing towards the root in each tree; furthermore, every such forest of rooted trees is a special graph. 

%our construction <=> trees rooted at nth vertex

We construct a mapping from special graphs to trees on $n$ vertices rooted at the $n$th vertex. 
For each special graph, construct an additional, $n$th vertex in the graph. For each of the original $n-1$ vertices that does not have an outedge, create a directed edge leading to the $n$th vertex. This gives a tree with all edges directed at the $n$th vertex (i.e. a tree rooted at the $n$th vertex).

%//trees with all edges directed at nth vertex <=> trees rooted at nth vertex

Note that this mapping is injective, as the vertices with edges directed toward the $n$th vertex are precisely those without outedges in the original graph, while all edges are preserved. 
On the other hand, the mapping is surjective over trees rooted at the $n$th vertex, because we can just direct all edges toward the root and remove the root and its edges. This must form a forest of rooted trees. 

%trees rooted at nth vertex <=> trees
Thus, a bijection exists between such directed graphs and rooted undirected trees on $n$ vertices. Furthermore, the number of trees rooted at a particular vertex (the $n$th vertex in this case) on $n$ labeled vertices is simply the number of trees on $n$ labeled vertices. Thus, to count the number of labeled trees on $n$ vertices, we need only to count the number of directed graphs defined above.

Since there are $n^{n-1}$ possible graphs, and exactly $1-p=\frac{1}{n}$ of the graphs correspond to labeled trees with $n$ vertices, there are a total of $n^{n-2}$ trees on $n$ vertices, as desired.
\end{proof}

\section {Generalizations}

In addition to the above result, we found two nice generalizations of the lemma. %is this line necessary?

%generations and extensions and whatever?
\begin {corollary}
Consider a variant of the graph generation process in Lemma~\ref{keylemma}. For each vertex $v$, indepedently perform the following process: Pick a vertex $w$ uniformly at random from the vertices of $G$. With probability $p_v$, call $v$ good and draw a directed edge between $v$ and $w$. The probability of a cycle in the graph is equal to the average of the probabilities of each vertex being good, or $\frac{\sum\limits_v p_v}{n}.$
\end {corollary}

\begin{proof}
As shown in the proof of the lemma, the probability of a cycle existing is equal to $\frac{1}{n}$ of the expected number of good vertices in the graph, or $\frac{\sum\limits_v p_v}{n}.$
\end{proof}

\begin {corollary}
Consider another variant of the graph generation in Lemma~\ref{keylemma}. Let $\pi$ be a probability distribution over the vertices. For each vertex $v$, indepedently perform the following process: Pick a vertex $w$ from the vertices of $G$ according to $\pi$. With probability $p$, call that vertex \textit{good} and draw a directed edge between $v$ and $w$. The probability of a cycle in the graph is equal to $p.$ %and independent of $\pi$
\end {corollary}

\begin{proof}
Proceed with strong induction as in the lemma. We wish to show that the probability of a cycle in the graph is equal to $\frac{1}{n}$ of the expected number of good vertices. The $k=0$ and $k=n$ cases remain the same. For all other values of $k$, note that the probability of each node in $G'$ having its outedge lead to a node in $G'$ is $\sum\limits_{v \in G'} \pi(v)$. However, since each node has equal probability of being good, the expected value of this probability, and therefore the expected value of a cycle in $G'$ by the inductive hypothesis, is $k(\frac{1}{n})(\sum\limits_{v \in G} \pi(v)) = \frac{k}{n}$, as desired.
\end{proof}

%more stuff

\section{Acknowledgements}
We would like to thank Po Shen Loh, Yan Zhang, and others for their guidance. %We would also like to thank Bessie the Cow for inspiration of this result. 

\bibliographystyle{plain}
\bibliography{cayley}
\end{document}